\documentclass[preprint,12pt]{elsarticle}
\usepackage{graphicx}
\usepackage[cp1251]{inputenc}
\usepackage{amssymb}
\usepackage{amsthm}
\usepackage{cmap}
\usepackage{picture}

\usepackage{multicol}

\usepackage{amsmath}% http://ctan.org/pkg/amsmath

\biboptions{sort&compress}

\newcommand{\sysn}{\left\{\begin{array}{rcl}}
\newcommand{\sysk}{\end{array}\right.}

\newtheorem{theorem}{Theorem}[section]
\newtheorem{lemma}[theorem]{Lemma}

\theoremstyle{example}

\newtheorem{proposition}[theorem]{Proposition}

\newtheorem{corollary}[theorem]{Corollary}

\theoremstyle{definition}
\newtheorem{definition}[theorem]{Definition}

\newtheorem{question}[theorem]{Question}
\journal{...}

\begin{document}

\title{Resolvability in products of spaces of small cardinality}

\author{Anton Lipin}

\address{Krasovskii Institute of Mathematics and Mechanics, \\ Ural Federal
 University, Yekaterinburg, Russia}

\ead{tony.lipin@yandex.ru}

\begin{abstract} We prove that:
	
	I. The product of any two regular isodyne spaces of cardinality $\omega_1$ is $\omega$-resolvable;
	
	II. The product of any $n + 2$ Hausdorff isodyne spaces of cardinality $\omega_n$ is $\omega$-resolvable.
\end{abstract}

%\tnotetext[label1]{The research has been supported by .}


\maketitle %%
% \linenumbers


\section{Introduction}

Suppose $\kappa$ is a cardinal.
A topological space $X$ is called $\kappa${\it -resolvable} iff there are $\kappa$-many pairwise disjoint dense subsets of $X$.
A $2$-resolvable space is called {\it resolvable}, whereas a crowded (i.e. without isolated points) not resolvable space is called {\it irresolvable}.
A space $X$ is called {\it maximally resolvable} iff $X$ is $\Delta(X)$-resolvable, where $\Delta(X)$ is the {\it dispersion character} of $X$, i.e. the smallest cardinality of a nonempty open subset of $X$.
An irresolvable space which contains no resolvable subspaces is called {\it hereditary irresolvable} \cite{Pavlov2007}.

Recently we proved that the product of each two countable crowded Hausdorff spaces is $\omega$-resolvable \cite{Lipin2025}.
Now we generalize the used approach in two different ways. It turns out that the condition of countability of factors can be slightly weakened, providing that they are regular (Theorem \ref{TReg}) or that we multiply more than two spaces (Theorem \ref{TCube}).

\section{Preliminaries}

We need the following classic results.

\begin{proposition}[J.G.~Ceder, \cite{Ceder1964}]\label{PCeder}
	If every nonempty open subset of a space $X$ contains a $\kappa$-resolvable subspace, then the space $X$ is $\kappa$-resolvable.
\end{proposition}

\begin{corollary}\label{CCeder}$\empty$
	Suppose $X$ and $Y$ are spaces and for all nonempty open sets $U \subseteq X$ and $V \subseteq Y$ there is a $\kappa$-resolvable subspace of $U \times V$.
	Then $X \times Y$ is $\kappa$-resolvable.
\end{corollary}

\begin{theorem}[Illanes, Bhaskara Rao, \cite{Illanes1996, BhaskaraRao2019}]\label{TIBR}
	If a space $X$ is $n$-resolvable for every $n < \omega$, then $X$ is $\omega$-resolvable.
\end{theorem}

A space $X$ is called {\it isodyne} iff all nonempty open subsets of $X$ have cardinality $|X|$.
Every nonempty space has a nonempty open isodyne subspace, thus by Corollary \ref{CCeder} one may reduce the research in resolvability of products to the case of isodyne factors.

The symbols $\bigsqcup, \sqcup$ in the following proof and throughout the paper mean the {\it disjoint union}, i.e. the union of pairwise disjoint sets.

\begin{lemma}\label{Pomega_to_irres}
	If an infinite isodyne space $X$ is not $\omega$-resolvable, then $X$ contains a hereditary irresolvable isodyne subspace of cardinality $|X|$.
\end{lemma}
\begin{proof}
	There is a nonempty open set $Y \subseteq X$ which contains no $N$-resolvable subsets for some $N < \omega$: otherwise $X$ would be $n$-resolvable for all $n < \omega$ and hence $\omega$-resolvable.
	For every $A \subseteq Y$ denote by $f(A)$ the biggest $n < N$ such that $A$ contains an $n$-resolvable subspace.
	
	Take $Z \subseteq Y$ such that $\Delta(Z) = |X|$ and $f(Z) = \min\{f(A) : A \subseteq Y, \Delta(A) = |X|\}$.
	Let us prove that $f(Z) = 1$.
	
	Denote $k = f(Z)$ and suppose $k > 1$.
	By the minimality of $k$, every nonempty open in $Z$ set contains a $k$-resolvable subspace, so $Z$ is $k$-resolvable.
	In particular, $Z = A \sqcup B$ for some dense in $Z$ sets $A$ and $B$.
	Neither one of $A$ and $B$ is $k$-resolvable, since otherwise $Z$ would be $k + 1$-resolvable.
	Now there are two cases:
	
	I. $\Delta(A) = |X|$. Since $A$ is not $k$-resolvable, there is a nonempty open in $A$ set $C$ such that $f(C) < k$.
	Since $\Delta(C) = |X|$, we obtain a contradiction.
	
	II. $\Delta(A) < |X|$, i.e. there is a nonempty open in $Z$ set $U$ such that $|A \cap U| < |X|$.
	Denote $E = B \cap U$. Clearly, $\Delta(E) = |X|$ and $E$ is not $k$-resolvable, hence there is a nonempty open in $E$ set $C$ such that $f(C) < k$, so we obtain a contradiction again.
\end{proof}

The following technique is essentially same with one introduced by O.~Pavlov in \cite{Pavlov2002}.

\begin{definition}
	Suppose $X$ is an isodyne space.
	We denote by $\mathrm{tr}(X)$ the set of all points $x \in X$ such that for every set $A \subseteq X$ of cardinality less than $|X|$ there is a set $B \subseteq X \setminus A$ of cardinality less than $|X|$ such that $x \in B'$.
\end{definition}

\begin{proposition}\label{PPavlov}
	If $X$ is an isodyne space of regular cardinality and the set $\mathrm{tr}(X)$ is dense in $X$, then $X$ is maximally resolvable.
\end{proposition}
\begin{proof}
	Denote $\kappa = |X|$.
	Choose any enumeration $\mathrm{tr}(X) \times \kappa = \{(x_\xi, \alpha_\xi) : \xi < \kappa\}$.
	Recursively on $\xi < \kappa$, choose pairwise disjoint sets $A_\xi \subseteq X$ of cardinality less than $\kappa$ in such a way that $x_\xi \in A_\xi'$.
	Construct $D_\alpha = \bigsqcup\{A_\xi : \alpha_\xi = \alpha\}$ for all $\alpha < \kappa$.
	Clearly, the sets $D_\alpha$ are dense in $X$ and pairwise disjoint.
\end{proof}

The following observation is the final step in the proofs of both our main theorems.

\begin{proposition}\label{PResFromFunc}
Suppose $X$ is a space, $\varphi : X \to \omega$ and every $x \in X$ is a limit point for the set $\{y \in X : \varphi(y) = \varphi(x) + 1\}$.
Then the space $X$ is $\omega$-resolvable.
\end{proposition}
\begin{proof}
Choose any function $c : \omega \to \omega$ with the property that the set $c^{-1}(n)$ is infinite for all $n \in \omega$.
Define $D_n = \{x \in X : c(\varphi(x)) = n\}$. Clearly, the sets $D_n$ are pairwise disjoint and every one of them is dense in $X$.
\end{proof}

\section{Resolvability in products of regular spaces of cardinality $\omega_1$}

Throughout this section the reader may assume $\kappa = \omega_1$.

\begin{definition}
	Suppose $X$ is a space and $\kappa \geq \omega$ is an infinite cardinal. We say that:
	\begin{itemize}
		\item a set $A \subseteq X$ is $\kappa$-{\it remote} iff $A \cap \overline{S} = \emptyset$ whenever $S \subseteq X \setminus A$ and $|S| < \kappa$;
		
		\item a disjoint family $\mathcal{H}$ of elements of $X$ is $\kappa$-{\it rare} iff all elements of $\mathcal{H}$ are $\kappa$-remote and $|\bigsqcup\mathcal{H}| < \kappa$;
		
		\item the space $X$ is $\kappa$-{\it disentangled} iff for every $\kappa$-rare family $\mathcal{H}$ there are pairwise disjoint open sets $O(A) \supseteq A$ for $A \in \mathcal{H}$.
	\end{itemize}
\end{definition}

Clearly, every open set is $\kappa$-remote for each cardinal $\kappa$.
Moreover, the family of all $\kappa$-remote subsets of a space $X$ is a topology on $X$.

The following proposition seems classic, but the author was unable to find it in the literature.

\begin{proposition}\label{PSeparatingCountableSets}
	Every regular space is $\omega_1$-disentangled.
\end{proposition}
\begin{proof}
	Suppose $\mathcal{H}$ is an $\omega_1$-rare family in some regular space $X$.
	
	Let us at first consider the case $|\mathcal{H}| = 2$, i.e. $\mathcal{H} = \{A_0, A_1\}$.
	Choose any indexing $A_0 \sqcup A_1= \{x_n : n < \omega\}$.
	Denote $U^0_0 = U^0_1 = \emptyset$, then suppose that $n < \omega$ and we are given open sets $U^n_0$, $U^n_1$ such that $U^n_0 \cap U^n_1 = \emptyset$ and $\overline{A_i \cup U^n_i} \cap A_{1-i} = \emptyset$ for $i < 2$.
	Take $j < 2$ such that $x_n \in A_j$. Since $X$ is regular, there is an open set $V_n$ such that $x_n \in V_n$ and $\overline{V_n} \cap \overline{A_{1-j} \cup U^n_{1-j}} = \emptyset$. Define $U^{n+1}_j = U^n_j \cup V_n$ and $U^{n+1}_{1-j} = U^n_{1-j}$.
	Clearly, the sets $O(A_i) = \bigcup_{n < \omega} U^n_i$ are as required.
	
	Now, if $\mathcal{H} = \{A_n : n < |\mathcal{H}|\}$, recursively on $n < |\mathcal{H}|$ choose open sets $O(A_n)$ and $W_n$ in such a way that $O(A_n) \cap W_n = \emptyset$, $A_n \subseteq O(A_n)$, $\bigsqcup_{k > n}A_k \subseteq W_n$ and $O(A_{n+1}) \cup W_{n+1} \subseteq W_n$.
	Clearly, the sets $O(A_n)$ are as required.
\end{proof}

\begin{definition}\label{DefCarving}
	Suppose $X$ is a space and $\kappa$ is a cardinal.
	We say that a disjoint family $\mathcal{A}$ of nonempty $\kappa$-remote subsets of $X$ is a $\kappa$-{\it carving} on the space $X$ iff
	\begin{enumerate}
		\item[(C1)] $|\bigsqcup \mathcal{H}| < \kappa$ whenever $\mathcal{H} \subseteq \mathcal{A}$ and $|\mathcal{H}| < \kappa$;
		
		\item[(C2)] for every nonempty open set $U \subseteq X$ the family $\{A \in \mathcal{A} : A \cap U \ne \emptyset\}$ has cardinality $\kappa$.
		In particular, $|\mathcal{A}| = \kappa$.
	\end{enumerate}
	
	We say that a space $X$ is $\kappa$-{\it carvable} iff there is a $\kappa$-carving on $X$.
\end{definition}

Clearly, if $\mathcal{A}$ is a $\kappa$-carving, $\mathcal{H} \subseteq \mathcal{A}$ and $|\mathcal{H}| < \kappa$, then $\mathcal{H}$ is $\kappa$-rare.

\begin{proposition}\label{PFork}
	If an isodyne space $X$ of regular uncountable cardinality is not $\omega$-resolvable, then $X$ contains a crowded $|X|$-carvable subspace.
\end{proposition}
\begin{proof}
Denote $\kappa = |X|$.
By Lemma \ref{Pomega_to_irres} the space $X$ contains a hereditary irresolvable subspace $Y$ such that $|Y| = \Delta(Y) = \kappa$.

The set $\mathrm{tr}(Y)$ cannot be dense in $Y$, because otherwise $Y$ would be maximally resolvable by Proposition \ref{PPavlov}.
Thus, there is a nonempty open in $Y$ set $Z \subseteq Y$ such that $Z \cap \mathrm{tr}(Y) = \emptyset$.
Let us show that $Z$ is $\kappa$-carvable.

Denote by $\mathcal{A}$ an arbitrary maximal disjoint family of $\kappa$-remote subsets of $Z$ of cardinality less than $\kappa$.
Clearly, $\mathcal{A}$ satisfies the condition (C1) of Definition \ref{DefCarving} of a $\kappa$-carving.
Let us prove that $\mathcal{A}$ satisfies the condition (C2) as well.
Suppose, on the contrary, that there is a nonempty open in $Z$ set $U$ such that the set $M = U \cap \bigsqcup \mathcal{A}$ has cardinality less than $\kappa$.
The set $U \setminus M$ is dense in $U$ and $U$ is irresolvable, so there is a nonempty open in $Z$ set $V \subseteq U$ such that $V \cap M = \emptyset$.

Since $V \cap \mathrm{tr}(Y) = \emptyset$, for every point $x \in V$ there is a set $R(x) \subseteq V$ such that $|R(x)| < \kappa$ and $x \notin S'$ whenever $S \subseteq V \setminus R(x)$ and $|S| < \kappa$.
Denote by $B_0$ any one-point subset of $V$, $B_{n + 1} = B_n \cup \bigcup_{x \in B_n}R(x)$ and $B = \bigcup_{n < \omega}B_n$.
Clearly, $B$ is $\kappa$-remote, $|B| < \kappa$ and $B$ intersects no elements of $\mathcal{A}$.
It is a contradiction with maximality of $\mathcal{A}$.
\end{proof}

\begin{proposition}\label{PSumOfCarvable}
	If $Z = X \sqcup Y$, where $X$ and $Y$ are open $\kappa$-carvable subspaces of a space $Z$, then the space $Z$ is $\kappa$-carvable as well.
\end{proposition}
\begin{proof}
	Clearly, if $\mathcal{A}$ is a $\kappa$-carving on $X$ and $\mathcal{B}$ is a $\kappa$-carving on $Y$, then $\mathcal{A} \sqcup \mathcal{B}$ is a $\kappa$-carving on $Z$.
\end{proof}

\begin{lemma}\label{LCarving}
	If a space $X$ is $\kappa$-carvable and $X^2$ is $\kappa$-disentangled, then $X^2$ is $\omega$-resolvable.
\end{lemma}
\begin{proof}
	Suppose $\mathcal{A}$ is $\kappa$-carving on the space $X$.
	Choose any enumeration $\mathcal{A} = \{A_\alpha : \alpha < \kappa\}$.
	For all $\alpha,\beta < \kappa$ denote $C_\alpha^\beta = A_\alpha \times A_\beta$.
	For all $\delta < \kappa$ denote
	$$L_\delta = \bigsqcup\limits_{\alpha < \delta} C_\alpha^\delta, \;\;
	R_\delta = \bigsqcup\limits_{\alpha > \delta} C_\alpha^\delta.$$
	
	For every set $M \subseteq X^2$ we denote $M^{-1} = \{(y,x) : (x,y) \in M\}$ and for every family $\mathcal{M}$ of subsets of $X^2$ we denote $\mathcal{M}^{-1} = \{M^{-1} : M \in \mathcal{M}\}$.
	
	\begin{picture}(360, 150)
		\put(105, 10){\vector(1,0){150}}
		\put(240, 0){$X$}
		
		\put(115, 0){\vector(0,1){150}}
		\put(105, 135){$Y$}
		
		\put(170, 0){\line(0, 1){150}}
		\put(190, 0){\line(0, 1){150}}
		\put(175, 0){$A_\delta$}
		
		\put(105, 65){\line(1, 0){150}}
		\put(105, 85){\line(1, 0){150}}
		\put(100, 70){$A_\delta$}
		
		\put(175, 70){$C_\delta^\delta$}
		\put(135, 70){$L_\delta$}
		\put(210, 70){$R_\delta$}
		\put(170, 30){$L_\delta^{-1}$}
		\put(170, 110){$R_\delta^{-1}$}
	\end{picture}
	\begin{center}
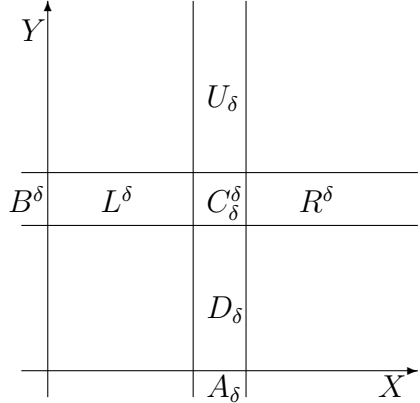
		Figure 1. A rough visualization of the defined sets
	\end{center}
	
	We claim that $\overline{R_\delta} = X \times \overline{A_\delta}$.
	Indeed, suppose an open set $P \times Q$ has a nonempty intersection with $X \times \overline{A_\delta}$, i.e. $Q \cap A_\delta \ne \emptyset$.
	Since $\mathcal{A}$ is a $\kappa$-carving, there is an ordinal $\alpha > \delta$ such that $A_\alpha \cap P \ne \emptyset$.
	Thus, $P \times Q$ has a nonempty intersection with $R_\delta$, which verifies our claim.
	
	In particular, $L_\delta \subseteq (R_\delta)'$.
	
	For all $\delta < \kappa$ we will construct disjoint families $\mathcal{L}_\delta$ and $\mathcal{R}_\delta$ of $\kappa$-remote subsets of $X^2$ and an injection $H_\delta : \mathcal{L}_\delta \to \mathcal{R}_\delta$ in such a way that the following four conditions hold:
	
	\begin{multicols}{2}
		\begin{itemize}
			\item[(L)] $\bigsqcup \mathcal{L}_\delta \subseteq L_\delta \subseteq \overline{\bigsqcup \mathcal{L}_\delta}$;
			
			\item[(R)] $\bigsqcup \mathcal{R}_\delta \subseteq R_\delta$;
		\end{itemize}
	\end{multicols}
	
	\begin{itemize}
		\item[(H)] $M \subseteq H_\delta(M)'$ for every $M \in \mathcal{L}_\delta$;
		
		\item[(UL)] for each  $\alpha < \delta$ and every $M \in \mathcal{R}_\alpha^{-1}$ we have $M \cap L_\delta \in \mathcal{L}_\delta$.
	\end{itemize}
	
	Suppose $\delta < \kappa$ and the families $\mathcal{L}_\gamma$, $\mathcal{R}_\gamma$ and injections $H_\gamma$ are constructed for all $\gamma < \delta$.
	
	Define $\mathcal{L}^*_\delta = \{M \cap L_\delta : M \in \mathcal{R}_\alpha^{-1}, \alpha < \delta\}$, $L_\delta^* = L_\delta \setminus \overline{\bigsqcup \mathcal{L}^*_\delta}$ and $\mathcal{L}_\delta = \mathcal{L}^*_\delta \sqcup \{L_\delta^*\}$.
	The family $\mathcal{L}_\delta$ is $\kappa$-rare, hence there are pairwise disjoint open sets $O_\delta(L) \supseteq L$ for $L \in \mathcal{L}_\delta$.
	Denote $H_\delta(L) = O_\delta(L) \cap R_\delta$ and $\mathcal{R}_\delta = \{H_\delta(L) : L \in \mathcal{L}_\delta\}$.
	It is easy to see that $\mathcal{L}_\delta$, $\mathcal{R}_\delta$ and $H_\delta$ are as required and the conditions (L,R,H,UL) are satisfied.
	
	\medskip
	
	Now let us unite all the families $\mathcal{L}_\delta$ and $\mathcal{L}_\delta^{-1}$ for $\delta < \kappa$ into one family $\mathcal{E}$
	and define the function $F$ on $\mathcal{E}$ in the following way:
	$$F(M) = \begin{cases}
		H^\delta(M), \; M \in \mathcal{L}_\delta; \\
		H^\delta(M)^{-1}, \; M \in \mathcal{L}_\delta^{-1}.
	\end{cases}$$
	
	It follows from the property (L) that $\mathcal{E}$ is disjoint and the set $E = \bigsqcup \mathcal{E}$ is dense in the set $\bigsqcup_{\alpha \ne \beta}C_\alpha^\beta$, which, in turn, is dense in $X^2$.
	Moreover, $F(M) \cap F(N) = \emptyset$ whenever $M,N \in \mathcal{E}$ and $M \ne N$.
	For all $M,N \in \mathcal{E}$ let us write $M \mapsto N$ iff $N \subseteq F(M)$.
	Thus, for every $N \in \mathcal{E}$ there is at most one $M \in \mathcal{E}$ such that $M \mapsto N$.
	
	We claim that the relation $\mapsto$ is well-founded.
	For every $\alpha,\beta < \kappa$ and each point $p \in C_\alpha^\beta$ denote $||p|| = \max\{\alpha, \beta\}$.
	Suppose $M \mapsto N$ and take arbitrary points $p \in M$ and $q \in N$.
	It is easy to see that $||p|| < ||q||$, which verifies the claim.
	
	It follows that we can define the function $\varphi : E \to \omega$ as follows (for every $q \in N \in \mathcal{E}$):
	$$\varphi(q) = \begin{cases}
		0, \; \text{there is no $M$ such that $M \mapsto N$}; \\
		\varphi(p) + 1, \; p \in M \mapsto N.
	\end{cases}$$
	
	Thus, by Proposition \ref{PResFromFunc} the subspace $E$ (and hence $X^2$) is $\omega$-resolvable.
\end{proof}

\begin{theorem}\label{TReg}
	The product of any two regular isodyne spaces of cardinality $\omega_1$ is $\omega$-resolvable.
\end{theorem}
\begin{proof}
	Suppose $X$ and $Y$ are regular isodyne spaces of cardinality $\omega_1$.
	We may assume that $X \cap Y = \emptyset$.
	By Corollary \ref{CCeder} it is sufficient to prove that the product $X \times Y$ contains an $\omega$-resolvable subspace.
	Clearly, if either one of $X$ and $Y$ is $\omega$-resolvable, then $X \times Y$ is $\omega$-resolvable as well.
	Else by Proposition \ref{PFork} $X$ contains some $\omega_1$-carvable subspace $X_1$ and $Y$ contains some $\omega_1$-carvable subspace $Y_1$.
	Consider the space $Z = X_1 \sqcup Y_1$, in which $X_1$ and $Y_1$ are open subsets.
	By Proposition \ref{PSumOfCarvable} the space $Z$ is $\omega_1$-carvable, so by Lemma \ref{LCarving} the space $Z^2$ is $\omega$-resolvable.
	Since $X_1 \times Y_1$ is an open subspace of $Z^2$, the space $X_1 \times Y_1$ is $\omega$-resolvable as well.
\end{proof}

Our proof does not work for spaces of cardinality $\omega_2$, because not all such spaces are $\omega_2$-disentangled.

\begin{question}
Is it true in ZFC that the product of two regular isodyne spaces of cardinality $\omega_2$ is $3$-resolvable?
\end{question}

If $2^{\omega_1} = \omega_2$, then such a product is maximally resolvable without any separation axioms at all \cite{Lipin2025}.

\begin{question}
Is it true in ZFC that the product $X \times Y$ is resolvable whenever $X$ and $Y$ are regular isodyne spaces of cardinalities $\omega$ and $\omega_1$ respectively?
\end{question}

The existence of two crowded spaces with irresolvable product is equiconsistent with the existence of a measurable cardinal \cite{BL1996,JSS2023}, but the smallest possible cardinality of such an example remains unknown. If CH holds (or at least $\mathrm{cf}(\frak{c}) = \omega_1$), then the product from the last question is maximally resolvable without any separation axioms \cite{Lipin2025}.

\section{Resolvability in products of three and more Hausdorff spaces}

\begin{definition}
	Suppose $I$ is a nonempty finite set and $\kappa$ is a cardinal.
	Let us denote by $\kappa_{1-1}^I$ the family of all injections from $I$ into $\kappa$.
	For every $p \in \kappa_{1-1}^I$ we define:
	\begin{itemize}
		\item $||p|| = \max\{p(i) : i \in I\}$;
		
		\item $m(p)$~--- the only $j \in I$ such that $p(j) = ||p||$.
	\end{itemize}
\end{definition}

A few times below we need to consider the cardinality of the ordinal $||p||$.
To avoid the confusing notation $|||p|||$, we denote $\#A = |A|$.

\begin{definition}
	Suppose $I$ is a nonempty finite set and $\kappa > 0$ is a cardinal.
	We say that a set $\prod_{i \in I} A_i \subseteq \kappa^I$ is a {\it linear fiber} in $\kappa^I$ with the {\it direction} $j \in I$ iff $A_j = \kappa$ and for all $i \in I \setminus \{j\}$ we have $|A_i| = 1$.
	We denote by $\mathbb{L}(\kappa^I)$ the family of all linear fibers in $\kappa^I$.
\end{definition}

Recall that for every cardinal $\kappa = \aleph_\alpha$ and each ordinal $\beta$ the notation $\kappa^{+\beta}$ means $\aleph_{\alpha + \beta}$.

\begin{lemma}\label{LBristling}
	Suppose $n < \omega$.
	For every infinite cardinal $\kappa$ and each set $I$ of cardinality $n + 2$ there is a function $F : \kappa_{1-1}^I \to \mathbb{L}(\kappa^I)$ such that for all $p \in \kappa_{1-1}^I$ and $L \in \mathbb{L}(\kappa^I)$:
	\begin{enumerate}
		\item $p \in F(p)$;
		
		\item $m(p)$ is not the direction of $F(p)$;
		
		\item $|F^{-1}(L)|^{+n} < \kappa$;
	\end{enumerate}
\end{lemma}
\begin{proof}
For every ordinal $\alpha < \kappa$ fix an arbitrary bijection $f_\alpha : \alpha \leftrightarrow |\alpha|$.

For every nonempty set $J \subseteq I$ and each injection $p : J \to \kappa$ we denote by $p'$ the injection $q : J \setminus \{m(p)\} \to \kappa$ defined as $q(i) = f_{||p||}(p(i))$.
Denote $p^{(0)} = p$ and $p^{(k+1)} = (p^{(k)})'$ for $k \leq |J|$.
Denote by $h(p)$ the only element of the domain of the function $p^{(|J| - 1)}$.

For every $p \in \kappa_{1-1}^I$ we denote by $F(p)$ the only linear fiber $L$ in $\kappa^I$ such that $p \in L$ and the direction of $L$ is $h(p)$.
Clearly, the conditions (1,2) are satisfied.
Let us take any linear fiber $L \in \mathbb{L}(\kappa^I)$ and show that $|F^{-1}(L)|^{+n} < \kappa$.

Denote by $j$ the direction of $L$.
Denote $E = I \setminus \{j\}$.
Take any point $p_0 \in L$ and denote by $r$ the restriction $p_0|_E$ of the function $p_0$ to the set $E$.
Clearly, this does not depend on $p_0$.
Moreover, for every $p \in F^{-1}(L)$ we have $m(p^{(k)}) = m(r^{(k)})$ for $k \leq n$ and $p^{(k)}|_E = r^{(k)}$ for $k \leq n + 1$.

For all $p \ne q$ in $F^{-1}(L)$ we have $p(j) \ne q(j)$ and hence $p^{(k)}(j) \ne q^{(k)}(j)$ for all $k \leq n + 1$.
For every $p \in F^{-1}(L)$ denote $\varphi(p) = p^{(n + 1)}(j)$.
Thus, the function $\varphi$ is an injection from $F^{-1}(L)$ to $\#||r^{(n)}||$ and hence $|F^{-1}(L)| \leq \#||r^{(n)}||$.
It follows from the definition of $'$ that there are at least $n$ cardinals between $||r^{(n)}||$ and $||r||$, so 
$\#||r^{(n)}||^{+n} \leq \#||r|| < \kappa$
and we obtain (3).
\end{proof}

\begin{proposition}\label{PBristling4}
	The function $F$ from Lemma \ref{LBristling} also satisfies the following condition for all $p,q \in \kappa_{1-1}^I$:
	\begin{enumerate}
		\item[{\rm(4)}] if $q \in F(p)$ and $||q|| > ||p||$, then the direction of $F(p)$ is $m(q)$.
	\end{enumerate}
\end{proposition}
\begin{proof}
	Denote by $j$ the direction of the linear fiber $F(p)$.
	For all $i \in I \setminus \{j\}$ we have $p(i) = q(i)$.
	Since $||q|| > ||p||$, we have $q(j) > ||p||$, so $m(q) = j$.
\end{proof}

\begin{theorem}\label{TCube}
	The product of any $n+2$ Hausdorff isodyne spaces of cardinality $\omega_n$ is $\omega$-resolvable for all $n < \omega$.
\end{theorem}
\begin{proof}
	Denote $I = n + 2$, $\kappa = \omega_n$ and suppose we are given Hausdorff isodyne topologies $\tau_i$ on $\kappa$ for all $i \in I$.
	
	Take the function $F$ from Lemma \ref{LBristling}.
	It follows from (3) that the set $F^{-1}(L)$ is finite for every $L \in \mathbb{L}(\kappa^I)$.
	Thus, we can choose pairwise disjoint neighborhoods $O(p)$ of the points $p$ for $p \in F^{-1}(L)$.
	Define $U(p) = O(p) \cap \{q \in L \cap \kappa_{1-1}^I : ||q|| > ||p||\}$.
	
	For all $p,q \in \kappa_{1-1}^I$ we write $p \mapsto q$ iff $q \in U(p)$.
	The relation $\mapsto$ is well-founded, since $p \mapsto q$ implies $||p|| < ||q||$.
	It follows from Proposition \ref{PBristling4} and choice of the sets $U(p)$ that for every $q \in \kappa_{1-1}^I$ there is at most one $p \in \kappa_{1-1}^I$ such that $p \mapsto q$.
	Therefore, we can define the function $\varphi : \kappa_{1-1}^I \to \omega$ in the following way:
	$$\varphi(q) = \begin{cases}
		0, \; \text{there is no $p$ such that $p \mapsto q$}; \\
		\varphi(p) + 1, \; p \mapsto q.
	\end{cases}$$
	
	Thus, by Proposition \ref{PResFromFunc} the subspace $\kappa_{1-1}^I$ is $\omega$-resolvable.
	Since $\kappa_{1-1}^I$ is dense, the space $\prod_{i \in I} (\kappa, \tau_i)$, is $\omega$-resolvable as well.
\end{proof}

\begin{question}
	Is there a Hausdorff (regular,\dots) isodyne space $X$, a cardinal $\kappa \geq 3$ and a natural number $n \geq 2$ such that the space $X^{n+1}$ is $\kappa$-resolvable, whereas the space $X^n$ is not $\kappa$-resolvable?
\end{question}

Although the techniques of this and previous sections are compatible, it seems insufficient to prove $\omega$-resolvability of the product of three regular isodyne spaces of cardinality $\omega_2$.

\medskip

\noindent\textbf{Acknowledgements.} The author is grateful to Vladislav Smolin and the referee for useful remarks.

The work was supported by the Foundation for the Advancement of Theoretical Physics and Mathematics “BASIS”.

%\bibliographystyle{model1a-num-names}

\end{document}